\documentclass[leqno]{article}
\usepackage{%refcheck,,showkeys
amsthm,amsmath,amssymb
}

\newcommand\Weps{\gge_\eps}
\newcommand\Heps{\diss_\eps}

\newcommand\loc{_{\rm{loc}}}
\newcommand\FF{E_\eps}
\newcommand\ope{{\mathcal A}}
\newcommand\gge{{\mathcal W}}
\newcommand\diss{{\mathcal H}}
\newcommand\ksue{\frac \kappa\eps}
\newcommand\emte{e^{-t/\eps}}

\newcommand\Rplus{\R^{\scriptscriptstyle +}}
\newcommand\spt{\Rplus \!\times\!\R^n}

\newcommand\spint\int
\newcommand\INTTO[1]{\int_0^{#1}\!\!\spint}
\newcommand\IT\INTTO
\newcommand\emt{e^{-t}}
\newcommand\ems{e^{-s}}
\newcommand\emgs{e^{-\psi(s)}}

\newcommand\de{\,dxdt}
\newcommand\eps\varepsilon
\newcommand\dedeltaz[1]{\frac{\partial}{\partial\delta}#1 \big\vert_{\delta=0}}

\newcommand\weps{w_\eps}
\newcommand\embed{\!\!\hookrightarrow\!\!}
\newcommand\jeps{J_\eps}
\newcommand\feps{F_\eps}
\newcommand\R{\mathbb R}
\newcommand\Rpiu{\mathbb R^+}
\newtheorem{theorem}{Theorem}[section]

\newtheorem{proposition}[theorem]{Proposition}
\newtheorem{lemma}[theorem]{Lemma}
\newtheorem{corollary}[theorem]{Corollary}

\newtheorem{problem}{Problem}

\theoremstyle{remark}
\newtheorem{remark}[theorem]{Remark}

\theoremstyle{definition}
\newtheorem{definition}[theorem]{Definition}

\newenvironment{remnot}{\par\noindent{\bf Remark on notation.}}
{\par\medskip}

\makeatletter
\def\namedlabel#1#2{\begingroup
\def\@currentlabel{#2}%
\label{#1}\endgroup
}
\makeatother

\date{}

\begin{document}

\title{{\bf A minimization approach \\ to hyperbolic Cauchy problems}}

\vskip 3 cm

\author{\\ \bf Enrico Serra, Paolo Tilli \\ \\
\small Dipartimento di Scienze Matematiche, Politecnico di Torino\\
\small Corso Duca degli Abruzzi, 24, 10129 Torino, Italy \\
\tt{\small enrico.serra@polito.it, paolo.tilli@polito.it}}

\maketitle

\begin{abstract} Developing an original idea of De~Giorgi,
we introduce a new and purely variational approach to the
Cauchy Problem for a wide class of defocusing hyperbolic
equations. The main novel feature is that the solutions are
obtained as limits of functions that minimize
suitable functionals  in space--time (where the
initial data of the Cauchy Problem serve as prescribed
\emph{boundary} conditions). This opens up the way to
new connections  between
the hyperbolic world and that of the Calculus of Variations.
Also dissipative equations can be treated.
Finally, we discuss several examples of equations that fit in this
framework, including nonlocal equations, in particular equations with
the fractional Laplacian.
\end{abstract}
\medskip



\vskip .5cm

\section{Introduction}

In this paper we introduce a new and purely variational approach
to the Cauchy Problem for a wide class of defocusing hyperbolic
PDEs having the formal structure
\begin{equation}
\label{eq::1}
w''(t,x)=-\nabla \gge\big(w(t,\cdot)\bigr)(x),\quad
(t,x)\in \R^+\times \R^n,
\end{equation}
with prescribed initial conditions
\begin{equation}
  \label{eq:condin2}
  w(0,x)=w_0(x),\quad
w'(0,x)=w_1(x).
\end{equation}
While a precise setting with all formal details and our main results
are given in Section~2, here we confine ourselves to a rather
informal description of our approach, focusing on
the main ideas that lie behind it and on the possible new
perspectives that it opens up, especially some new connections
between the variational world and hyperbolic PDEs of the kind \eqref{eq::1}.

In \eqref{eq::1}, $\nabla\gge$ is the G\^ateaux derivative of
a functional (e.g. one from the Cal\-cu\-lus of Variations) $\gge:W\to [0,\infty)$, where
$W$ is some Banach space of functions in $\R^n$, typically a Sobolev space.
If, for instance,  $\gge(u)=1/2\int |\nabla u|^2\,dx$ is the Dirichlet
integral and $W=H^1(\R^n)$ then, formally, $-\nabla \gge(u)=\Delta u$,
and \eqref{eq::1} reduces to the wave equation $w''=\Delta w$, much
in the same spirit as the heat equation $u'=\Delta u$ is the
\emph{gradient flow} of the Dirichlet integral. Thus, in a sense,
\eqref{eq::1} can be considered as a ``second order gradient flow'' for
the functional $\gge$.

Our aim is to initiate and try to develop a rather general program,
suggested by De~Giorgi in \cite{DG} (see also \cite{DGopere}), that
offers a new, purely variational approach to equations of the kind \eqref{eq::1},
possibly with the addition of a dissipative term (see below).
We alert the reader that in this paper the term ``variational'' refers, in the spirit of De~Giorgi, to
\emph{minimization}, rather than Critical Point Theory.

The main idea, the abstract counterpart to a specific conjecture
stated in  \cite{DG} and discussed in \cite{Noi},
is to associate
with the abstract evolution equation \eqref{eq::1}
the functional
\begin{equation}
  \label{eq:deffeps}
  \feps(w)=\frac{\eps^2}{2}\int_0^\infty\!\!\!\int_{\R^n} \emte
 |w''(t,x)|^2 \,dxdt
+
\int_0^\infty\emte
\gge(w(t,\cdot))\,dt.
\end{equation}
This functional is to be minimized, for fixed $\eps>0$,
among all functions $w(t,x)$ in spacetime $\R^+\times\R^n$
subject to the constraints \eqref{eq:condin2}, which now
play the role of \emph{boundary} conditions.
Assuming the existence of an absolute minimizer $w_\eps$, the
Euler--Lagrange equation of \eqref{eq:deffeps} formally reads
\[
\eps^2 \bigl(e^{-t/\eps} w''_\eps\bigr)''+ e^{-t/\eps}\nabla
\gge\bigl(w_\eps(t,\cdot)\bigr)(x)=0,
\]
that is, the fourth--order in time equation
\begin{equation}
\label{4ord}
\eps^2 w''''_\eps-2\eps w'''_\eps+w''_\eps
+
\nabla
\gge\bigl(w_\eps(t,\cdot)\bigr)(x)=0.
\end{equation}
The connection with \eqref{eq::1} is clear:
letting $\eps\downarrow 0$, one formally obtains \eqref{eq::1} in the limit. This motivates
the following
\begin{problem}[De Giorgi, \cite{DG,DGopere}]\label{prob1}
Let $w_\eps$ be a minimizer of $\feps$ in \eqref{eq:deffeps},
subject to the boundary conditions \eqref{eq:condin2}. Investigate
the existence of a limit
 function
 \begin{equation}
 \label{wlimite}
 w(t,x)=\lim_{\eps\to 0^+} w_\eps(t,x),
 \end{equation}
and see if it solves the Cauchy Problem \eqref{eq::1}$\&$\eqref{eq:condin2}.
\end{problem}
In its generality, as long as the structure of the
functional $\gge$ is unknown, this may sound a little vague.
In fact, in \cite{DG} De~Giorgi raised this general question taking cue
from a precise conjecture in a particular case, namely when
\[
\gge(w)=\frac 1 2\int_{\R^n} |\nabla w(x)|^2\,dx
+
\frac 1 p\int_{\R^n} |w(x)|^p\,dx\quad
(p\geq 2)
\]
and \eqref{eq::1} becomes the nonlinear wave equation
\[
w''=\Delta w-w |w|^{p-2} \quad
(p\geq 2).
\]
In this particular case, Problem~\ref{prob1} has an affirmative answer, \cite{Noi}.
As we will show, however, much can be said on Problem~1 under very mild
assumptions on $\gge$,
and a robust theory can be built that provides several a~priori estimates on the
minimizers $w_\eps$. In some cases, basically when $\gge(w)$ is quadratic in the
highest order derivatives of $w$, Problem~1 can be completely solved without any other assumption. In \emph{all} cases,
however, up to subsequences the limit \eqref{wlimite} always exists and the estimates
on $w_\eps$ entail the fulfillment of \eqref{eq:condin2}. When \eqref{eq::1} is highly nonlinear,
the general estimates still apply, but additional work is needed to get
stronger compactness on $w_\eps$ and possibly obtain \eqref{eq::1} in the limit
(of course such further estimates, if any, will depend on the particular structure
of $\gge(w)$, and should be obtained \emph{ad hoc} on a case--by--case basis).

The variational approach suggested by Problem~\ref{prob1} is by genuine minimization,
a completely new and unconventional feature, when it comes to \emph{hyperbolic} equations.
The typical case is when $\gge$ is a convex (lower semicontinuous, etc.) functional of
the Calculus of Variations
(possibly depending on $x$, $w$ and some of its \emph{spatial} derivatives): in this
case $\feps$ in \eqref{eq:deffeps} inherits the good properties of $\gge$, and the existence
of  $w_\eps$ (a minimizer of $\feps$ subject to \eqref{eq:condin2}) is not an issue.
Moreover, one may try to exploit several powerful techniques such as the theory
of regularity for minimizers to get strong compactness on
$w_\eps$ and pass to the limit in \eqref{wlimite}.

We believe that these features are a major point of interest of
the present work. Indeed on the one hand our results provide a
new, general starting point for the investigation of a wide class
of hyperbolic problems, and on the other they allow one to use
methods (coming from the elliptic theory) that have never been
applied before in this context. Thus, our framework might
hopefully help in shedding new light on several long--standing
open problems in the theory of nonlinear hyperbolic equations.

We also point out that although the fourth order equation \eqref{4ord} has the structure of a singularly perturbed equation, this
fact is never used in our results, that are simply based on the properties of minimizers of the functional $F_\eps$. For
instance, no estimates on the third and fourth order derivatives are required.
\medskip

Our approach also works with an extra (dissipative) term
in the right hand side of \eqref{eq::1}, namely
\begin{equation}
\label{eq::2}
w''(t,x)=-\nabla \gge\big(w(t,\cdot)\bigr)(x)
-\nabla \diss\big(w'(t,\cdot)\bigr)(x),\quad
(t,x)\in \R^+\times \R^n
\end{equation}
where $\diss:H\to [0,+\infty)$ is a G\^ateaux differentiable
functional, defined on a suitable Hilbert space $H\embed L^2(\R^n)$.
For the sake of simplicity, contrary to $\gge$, we will assume that $\diss$ is a \emph{quadratic} form
on $H$. Note that, while $\nabla\gge$ is computed
at $w$, $\nabla\diss$ is computed at $w'$:
if, for instance, both $\gge$ and $\diss$ are the Dirichlet integral, then \eqref{eq::2}
reduces to the strongly damped wave equation $w''=\Delta w+\Delta w'$.
The reader is suggested to look at Section \ref{esempi}, where we discuss
several examples of hyperbolic problems
(with or without dissipative terms) that fit into our scheme.

For equations with dissipative terms the counterpart to Problem 1 is

\begin{problem}[Dissipative case]\label{prob2}
Let $w_\eps$ be a minimizer of the functional
\begin{equation}
  \label{eq:deffepsdis}
\frac{\eps^2}{2}\int_0^\infty\!\!\!\int_{\R^n} \emte
 |w''(t,x)|^2 \,dxdt
+
\int_0^\infty\emte
\bigl\{
\gge(w(t,\cdot))
+
\eps
\diss(w'(t,\cdot))
\bigr\}
\,dt
\end{equation}
subject to the boundary conditions \eqref{eq:condin2}. Investigate
the existence of a limit for $\weps$ as in \eqref{wlimite},
and see if it solves the Cauchy Problem \eqref{eq::2}$\&$\eqref{eq:condin2}.
\end{problem}
As before, the functional \eqref{eq:deffepsdis} relates to \eqref{eq::2} via
its Euler--Lagrange equation
\[
\eps^2 \bigl(e^{-t/\eps} w''_\eps\bigr)''+ e^{-t/\eps}
\nabla
\gge\bigl(w_\eps(t,\cdot)\bigr)(x)
-
\eps
\bigl(
e^{-t/\eps}
\nabla
\diss\bigl(w'_\eps(t,\cdot)\bigr)(x)
\bigr)'=0,
\]
namely,
\[
\eps^2 w''''_\eps-2\eps w'''_\eps+w''_\eps
+
\nabla
\gge\bigl(w_\eps(t,\cdot)\bigr)(x)
+
\nabla
\diss\bigl(w'_\eps(t,\cdot)\bigr)(x)
-\eps
\bigl(
\nabla
\diss\bigl(w'_\eps(t,\cdot)\bigr)(x)\bigr)'\!\! =0
\]
which, formally, reduces to \eqref{eq::2} when $\eps\downarrow 0$.

Also in the dissipative cases our results provide estimates for the minimizers $\weps$, existence of a limit $w$,
and in general all the properties described above.
\medskip

A further point of interest is that, as is well known, the energy
\[
{\mathcal E}(t)=
\frac 1 2\int_{\R^n} |w'(t,x)|^2\,dx+
\gge\bigl(w(t,\cdot)\bigr)
\]
is formally preserved by the solutions of equation \eqref{eq::1}, while for equation \eqref{eq::2} the presence of dissipative
terms entails that the preserved quantity is
\[
{\mathcal E}(t)+
2\int_0^t \diss\bigl(w'(t,\cdot)\bigr)\,dt.
\]
Generally, however,  energy conservation is purely formal, since weak solutions are not regular enough
to justify the computations needed in its proof. Our solutions are no exception, but in all cases they satisfy the ``energy
inequalities''
\[
{\mathcal E}(t) \le {\mathcal E}(0) \quad\text{and}\quad {\mathcal E}(t)+
2\int_0^t \diss\bigl(w'(t,\cdot)\bigr)\,dt \le {\mathcal E}(0).
\]
for equations \eqref{eq::1} and \eqref{eq::2} respectively.

Finally, we point out that our results are stated for functions defined in the whole of $\R^n$. This choice is motivated
as this is a model case of particular interest. However our results hold, without significative changes, also in different contexts,
for instance for functions defined on an open subset $\Omega$ of $\R^n$ with Dirichlet or Neumann conditions imposed on $\partial \Omega$.
\medskip

The paper is organized as follows. The main results are stated in Section 2 and proved in sections 5 and 6. Section 3 
contains preliminary results and Section 4 is devoted to the key argument for the construction of the a priori estimates. Finally, several examples
are reported in Section 7.
\medskip

\begin{remnot}
Throughout the paper, a prime as in $v'$, $v''$ etc.
denotes partial
differentiation with respect to the time variable $t$.
For functions defined in spacetime we will write freely $u(t,x)$ or $u(t)$.
So if
$u(t, \cdot)$ is an element of a space $X$ and ${\cal G}$ is a functional on $X$, we will write indifferently
${\cal G}(u(t, \cdot))$ or ${\cal G}(u(t))$.
Moreover,  through the rest of the paper
symbols as $\int v\,dx$ will always
denote spatial integrals extended to the whole
of $\R^n$, and short forms such as $L^2$, $H^1$ etc. will  denote $L^2(\R^n)$, $H^1(\R^n)$ etc. 
Finally,  $\langle\cdot,\cdot\rangle$ will denote the
duality pairing between a Banach space $X$ and its dual $X'$, the space $X$ being
clear from the context.
\end{remnot}

\section{Functional setting and main results}
The functional $\feps(w)$ to be minimized, subject to the boundary conditions \eqref{eq:condin2},
is defined by \eqref{eq:deffeps} in the non--dissipative case, and by \eqref{eq:deffepsdis}
in the dissipative
case. We shall treat  the
two cases simultaneously, by letting
\begin{equation}
\begin{split}
  \label{eq:deffepsmisto}
  \feps(w)=&
\frac{\eps^2}{2}\int_0^\infty\!\!\!\int_{\R^n} \emte
 |w''(t,x)|^2 \,dxdt\\
&+
\int_0^\infty\emte
\bigl\{
\gge(w(t,\cdot))
+
\kappa\eps
\diss(w'(t,\cdot))
\bigr\}
\,dt\qquad (\kappa\in\{0,1\}),
\end{split}
\end{equation}
where the parameter $\kappa\in\{0,1\}$ plays the role of an on/off variable.
Dealing with Problem~\ref{prob2} (dissipative case) one should let
$\kappa=1$, while dealing with Problem~\ref{prob1} (non--dissipative case)
one should let $\kappa=0$ and ignore the functional
$\diss$.

Concerning the functionals $\gge$ and $\diss$, we make the following assumptions:
\begin{itemize}
\item[\bf(H1)] The functional
$\gge:L^2\to [0,+\infty]$ is lower semicontinuous in the weak topology, i.e.,
\begin{equation}
\label{lsci} \gge(v)\leq \liminf_{k\to\infty} \gge(v_k)\quad
\text{whenever $v_k\rightharpoonup v$ in $L^2$.}
\end{equation}
Moreover we assume that $\gge(v)<\infty \iff v\in W$,
a Banach space with
\begin{equation}
\label{densembed}
C^\infty_0\embed W\embed L^2\qquad\text{(dense and continuous inclusions).}
\end{equation}

We also assume that
$\gge$ is G\^ateaux differentiable on $W$, and that its
derivative $\nabla\gge:W\to W'$ satisfies the estimate
\begin{equation}
\label{stimagattheta} \Vert \nabla \gge(v)\Vert_{W'}
\leq C\bigl(1+\gge(v)^\theta\bigr),
\quad
C\geq 0,\quad\theta\in (0,1),\quad\forall v\in W.
\end{equation}

\item[\bf(H2)] If $\kappa=1$, we assume that
$\diss:L^2\to [0,+\infty]$ is a quadratic functional
\begin{equation}\label{eq:strutD}
\diss(v)=
\begin{cases}
\frac 1 2 B(v,v)&\text{if $v\in H$,}\\
+\infty &\text{if $v\in L^2\setminus H$}
\end{cases}
\end{equation}
where $B:H\times H\to \R$ is a symmetric, bounded, nonnegative bilinear form
on a Hilbert space $H$ with the  norm
$\Vert v\Vert_H^2=\Vert v\Vert_{L^2}^2+2\diss(v)$, and
 such that
\begin{equation}
\label{HembedL2}
C^\infty_0\embed H\embed L^2\qquad\text{(dense and continuous inclusions).}
\end{equation}
If $\kappa=0$, for definiteness we set $\diss\equiv 0$ and $H=L^2$.
\end{itemize}

\begin{remark}\label{rema1}
If $\nabla^k v$ denotes the tensor of all $k$--th partial derivatives of $v$,
a Dirichlet--like functional
\[
\gge(v)=\frac 1 p\int_{\R^n} \left| \nabla^k v(x)\right|^p\,dx\quad
\text{($p> 1$)}
\]
satisfies assumption (H1) with $W$ the Banach space of all $L^2$ functions $v$
such that $\nabla^k v\in L^p$, endowed with its natural norm. Since
\begin{equation}\label{gradiente}
\langle \nabla \gge(v),\eta\rangle=
\int_{\R^n} \left| \nabla^k v(x)\right|^{p-2}\nabla^k v(x)\cdot \nabla^k\eta(x)\,dx,
\quad
v,\eta\in W,
\end{equation}
we see that \eqref{stimagattheta} holds with $\theta=1-1/p$.
In view of the embeddings
\eqref{densembed}, the term $\nabla \gge(w(t, \cdot))$ in equations \eqref{eq::1} and \eqref{eq::2},
as a distribution (note that $W'\embed {\mathcal D}'$ by
\eqref{densembed}), acts as a differential operator (linear when $p=2$) of order $2k$. Note also
that the functional $\gge$ need not be convex. 
\end{remark}

\begin{remark}\label{remtypicalH}
The typical functional $\diss$ fulfilling (H2) has the form
\begin{equation}\label{typicalH}
\diss(v)=
\frac 1 2
\sum_{j\in S} \int_{\R^n}
\left\vert \partial^j v \right\vert^2\,dx
\end{equation}
where $S\subset {\mathbb N}^n$ is any finite set of multi--indices and $\partial^j$ denotes partial differentiation.
Here $H$ is
the space of those $v\in L^2$ such that $\diss(v)<+\infty$, and
$\nabla H(v)$, as a distribution (note that $H'$ is a space of distributions by
\eqref{HembedL2}) is the differential operator
$\sum_{j\in S}(-1)^{|j|} \partial^{2j}$.
\end{remark}

\begin{remark}\label{rema2}
Assumptions (H1) and (H2) are \emph{additively stable}. More precisely, if $\gge_i:L^2\to
[0,\infty]$ ($i=1,2$) are two functionals each satisfying (H1) (with Banach
spaces $W_i$, constants $\theta_i$ etc.), then the sum $\gge=\gge_1+\gge_2$
still satisfies (H1), now with $W=W_1\cap W_2$ normed by $\Vert\cdot\Vert_W=
\Vert\cdot\Vert_{W_1}+\Vert\cdot\Vert_{W_2}$
(this makes sense, in view of \eqref{densembed}).
In particular, by Young inequality,
\eqref{stimagattheta} will hold true with $\theta=\max\{\theta_1,
\theta_2\}$.

Finally, a similar argument applies to (H2).
\end{remark}

\begin{theorem}[non--dissipative case]\label{mainteo1}
Given $w_0,w_1\in W$ and $\eps\in(0,1)$, under assumption (H1) the
functional $\feps$ defined in \eqref{eq:deffeps} has a minimizer
$w_\eps$ in the space $H^2_{\rm{loc}}([0,\infty);L^2)$
 subject to
\eqref{eq:condin2}. Moreover:
\begin{enumerate}
\item[(a)] Estimates. There exists a constant $C$, independent of
$\eps$,  such that
\begin{equation}
  \label{eq:stimaapriori}
  \int_\tau^{\tau+T} \gge\bigl(w_\eps(t,\cdot)\bigr)\,dt
\leq CT\quad \forall\tau\geq 0,\quad\forall T\geq \eps,
\end{equation}
\begin{equation}
  \label{eq:stimaaprioriprime}
\int_{\R^n} |w_\eps'(t,x)|^2\,dx\leq C,\quad \int_{\R^n}
|w_\eps(t,x)|^2\,dx\leq C(1+t^2) \quad\forall t\geq 0,
\end{equation}
\begin{equation}
  \label{eq:stimaaprioridue}
  \left\Vert w''_\eps
\right\Vert_{L^\infty(\Rpiu;W')} \leq C.
\end{equation}

\item[(b)] Convergence. Every sequence  $w_{\eps_i}$ (with
$\eps_i\downarrow 0$) admits a subsequence which is convergent,
in the weak topology of $H^1((0,T);L^2)$ for every $T>0$, to a function
$w$ such that
\begin{align}
\label{eq:welleinfinito1} w\in H^1\loc([0,\infty);L^2),\quad
w'\in L^\infty(\R^+;L^2),\quad w''\in L^\infty(\R^+;W').
\end{align}
Moreover, $w$ satisfies the initial conditions \eqref{eq:condin2}.
\item[(c)] Energy inequality. Letting
\begin{equation}
  \label{eq:energiaonde}
  {\mathcal E}(t)=\frac 1 2\int_{\R^n}
|w'(t,x)|^2\,dx+ \gge\bigl(w(t,\cdot)\bigr),
\end{equation}
the function $w(t,x)$ satisfies the energy inequality
\begin{equation}
  \label{eq:quasiconservation}
  {\mathcal E}(t)\leq {\mathcal E}(0)=
\frac 1 2\int_{\R^n} |w_1(x)|^2\,dx+\gge(w_0) \quad\text{for a.e.
$t>0$.}
\end{equation}
\end{enumerate}
\end{theorem}

\begin{theorem}[dissipative case]\label{mainteo2}
Given $w_0 \in W$, $w_1\in W\cap H$ and $\eps\in(0,1)$, under
assumptions (H1) and (H2) the functional $\feps$ defined in
\eqref{eq:deffepsdis} has a minimizer $w_\eps$, in the space
$H^2_{\rm{loc}}([0,\infty);L^2)$
 subject to
\eqref{eq:condin2}. Moreover, all claims of Theorem~\ref{mainteo1}
apply, with the following extensions and modifications:
\begin{enumerate}
\item[(a)] The additional estimate
\begin{equation}
\label{provvD}
 \int_0^\infty \diss(w_\eps'(t))\, dt \le C
\end{equation}
holds true, while \eqref{eq:stimaaprioridue} should be replaced with
\begin{equation}
  \label{eq:stimaaprioriduebis}
  \left\Vert  w''_\eps
\right\Vert_{L^\infty(\Rpiu;W')+L^2(\Rpiu;H')} \leq C.
\end{equation}
\item[(b)] The part on $w''$ in \eqref{eq:welleinfinito1} should be
replaced with
\begin{equation}
\label{wsecondobis}
w''\in L^\infty(\R^+;W')+L^2(\Rpiu;  H').
\end{equation}
Moreover, the convergence $w_\eps'\to w'$ holds in a stronger sense, namely
\begin{equation}
\label{convL2H}
w_\eps'\rightharpoonup w'\quad\text{weakly in $L^2((0,T);H)$, for every $T>0$.}
\end{equation}

\item[(c)] With the same ${\mathcal E}(t)$, the inequality \eqref{eq:quasiconservation}
is replaced with
\begin{equation}
  \label{eq:quasiconservation2}
  {\mathcal E}(t)+2\int_0^t \diss(w'(t,\cdot))\,dt\leq {\mathcal E}(0)
  \quad\text{for
a.e. $t>0$.}
\end{equation}
\end{enumerate}
\end{theorem}
Observe that, under so general assumptions as in Theorem~\ref{mainteo1}
(or~\ref{mainteo2}), we do not claim that the limit function $w$
satisfies \eqref{eq::1} (or \eqref{eq::2}). On the other hand, to our knowledge
there are no counterexamples that rule out this possibility.
Of course, to
perform this step (by which one would completely solve Problem~\ref{prob1} or
\ref{prob2}) one should obtain extra estimates
exploiting the particular structure of the
functional $\gge$, on a case by case basis. In some cases,
however, the estimates of Theorem~\ref{mainteo1} (or \ref{mainteo2} if
$\kappa=1$) are enough to pass to the limit in the main equation,
as the following result illustrates.

\begin{theorem}\label{mainteo3}
Assume that, for some real number $m >0$,
\begin{equation}
\label{struct2} \gge(v)= \frac 1 2 \Vert v\Vert_{\dot H^m}^2
+ \sum_{0\leq k<m} \frac {\lambda_k}{p_k}
\int_{\R^n} \left| \nabla^k v(x)\right|^{p_k} \,dx\qquad
 (\lambda_k\geq 0, \quad p_k>1).
\end{equation}
Then assumption (H1) is fulfilled,
if $W$ is the space of those
$v\in H^m$ with $\nabla^k v\in L^{p_k}$ $(0\leq k<m)$ endowed with its
natural norm.

Moreover, the limit function $w$ obtained via Theorem~\ref{mainteo1} (or \ref{mainteo2}
if $\kappa=1$)) solves, in the sense of distributions, the hyperbolic equation~\eqref{eq::1} (or \eqref{eq::2} if
$\kappa=1$).
\end{theorem}

\begin{remark}
In \eqref{struct2}, as usual, $\dot H^m$ is  the $L^2$
norm of $|\xi|^m \hat v(\xi)$, where $\hat v$ is the Fourier transform of $v$. The typical case is when
$m$ is integer, so that $\Vert v\Vert_{\dot H^m}^2$ reduces to $\Vert\nabla^m v\Vert_{L^2}^2$.
In this case (see Remark \ref{rema1}) the first term in \eqref{struct2} gives rise to a differential operator
of order $2m$ in the equations \eqref{eq::1} and \eqref{eq::2}.

On the other hand, in \eqref{struct2} $m$ may fail to be integer. In this case,
however, one can interpret the distribution $\nabla\gge(w)$ in \ref{eq::1} or \eqref{eq::2}
as a \emph{fractional} differential operator: this enables us to treat,
for instance, equations with the \emph{fractional Laplacian} (see Example~\ref{exFL}).
\end{remark}

Several variants are possible in the same spirit.
For instance, one may introduce nonconstant coefficients in \eqref{struct2}
(and possibly exploit G{\aa}rding--type inequalities to make $\gge(v)$
coercive), or consider more general lower--order terms with suitable
convexity and growth assumptions (e.g. powers of single partial derivatives as in \eqref{typicalH}).
Indeed, the central assumption
is that $\gge$ be quadratic (and coercive) in the highest
order terms, which makes the hyperbolic PDEs \eqref{eq::1} and \eqref{eq::2} quasilinear.

We end this section by discussing some consequences of assumption (H1) and (H2)
which will be used in the sequel.
First, \eqref{stimagattheta} implies the linear control
\begin{equation}
\label{stimagat}
\Vert \nabla \gge(v)\Vert_{W'}\leq C\left(1+\gge(v)\right)\quad
C\geq 0,\quad\forall v\in W.
\end{equation}
Moreover, \eqref{stimagattheta} entails Lipschitz continuity of
$\gge$ along rays, as follows.
Given $a,\overline{b}\in W$ with $\Vert \overline{b}\Vert_W=1$,
the function $f(\lambda)=\gge(a+\lambda \overline{b})$ 
is differentiable and \eqref{stimagattheta} gives $|f'|\leq C (1+f^\theta)$.
From well known variants of the Gronwall Lemma, one has $f(\lambda)\leq C
\bigl(1+f(0)+ \lambda^{1/(1-\theta)}\bigr)$
and so
\begin{equation}
\label{growth}
\sup_{[a,a+b]} \gge \leq C \left(1+
\gge(a)+\Vert b\Vert_W^{\frac 1{1-\theta}}\right),\quad
\forall a,b\in W
\end{equation}
where $[a,a+b]$ is the segment in $W$ from $a$ to $a+b$.
Combining with \eqref{stimagat},
\begin{equation}
\label{growth3}
\sup_{[a,a+b]} \Vert \nabla \gge\Vert_{W'} \leq
C\left(1+ \gge(a)+\Vert b\Vert_W^{\frac 1{1-\theta}}\right),\quad
a,b\in W.
\end{equation}
Then, from Lagrange mean value theorem, for every $\delta\not=0$
\begin{equation*}
\left\vert
\frac {\gge(a+\delta b)-\gge(a)}{\delta}
\right\vert
\leq
\Vert b\Vert_W
\times
\sup_{[a,a+\delta b]} \Vert \nabla \gge\Vert_{W'}
\end{equation*}
and combining with \eqref{growth3},
\begin{equation}
\label{stimalip}
\left\vert
\frac {\gge(a+\delta b)-\gge(a)}{\delta}
\right\vert
\leq
C \Vert b\Vert_W \times
\left(1+ \gge(a)+\delta^{\frac 1{1-\theta}}\Vert b\Vert_W^{\frac 1{1-\theta}}\right),
\end{equation}
a quantitative bound for the Lipschitz constant of $\gge$. Thus, in particular,
\begin{equation}
\label{growth2}
\gge(a+\delta b) \leq
\gge(a)+
C\delta \Vert b\Vert_W \times
\left(1+ \gge(a)+\delta^{\frac 1{1-\theta}}\Vert b\Vert_W^{\frac 1{1-\theta}}\right).
\end{equation}

Finally, assumption (H2) entails that
$\diss$ is differentiable in $H$, with
\begin{equation}
\label{gradDDE}
\langle \nabla \diss(v), \eta\rangle=
B(v,\eta)
,\quad
\left\Vert \nabla \diss(v)\right\Vert_{H'}\leq
\sqrt{2 \diss(v)},\quad
v,\eta\in H.
\end{equation}
Moreover, $\diss$ is a fortiori weakly lower semicontinuous in $L^2$, namely
\begin{equation}
\label{lsci2} \diss(v)\leq \liminf_{k\to\infty} \diss(v_k)\quad
\text{whenever $v_k\rightharpoonup v$ in $L^2$.}
\end{equation}

\section{Existence of minimizers  and preliminary estimates}\label{sec2}

Since the space $H^2_{\text{loc}}([0,\infty);L^2)$ is
invariant under time dilations $t\mapsto \eps t$,
it is convenient
to introduce the simpler
functional
\begin{equation}
  \label{eq:defjeps}
  \jeps(u)=\int_0^\infty \emt\left(
\int \frac {|u''(t,x)|^2}{2\eps^2} \,dx
+\gge(u(t))+\ksue \diss(u'(t))\right)\,dt,
\end{equation}
equivalent to $\feps$ in \eqref{eq:deffepsmisto}
in that
$\feps(w)=\eps \jeps(u)$
whenever
$u,w\in H^2_{\text{loc}}([0,\infty);L^2)$
are related by the change of variable
$u(t,x)=w(\eps t,x)$. Of course,
the boundary conditions in \eqref{eq:condin2} must be scaled
accordingly, namely as in \eqref{eq:condiniz}.

The existence of minimizers $\weps$ for $\feps$ (as claimed in Theorems \ref{mainteo1}
and \ref{mainteo2}) then
follows from  the existence of minimizers $u_\eps$ for $\jeps$ and
\begin{equation}
  \label{eq:uv}
  u_\eps(t,x)=\weps(\eps t,x),\quad t\geq 0,\quad x\in{\R^n}.
\end{equation}

\begin{lemma}\label{existence2}
Given $\eps\in (0,1)$ and $w_0, w_1\in W$ (with $w_1\in W\cap H$ if $\kappa=1$) the
functional $\jeps$ has an absolute minimizer $u_\eps$, in the
class of those functions $u\in H^2_{\text{loc}}([0,\infty);L^2)$
satisfying the boundary conditions
\begin{equation}
  \label{eq:condiniz}
  u(0)=w_0,\quad u'(0)=\eps w_1.
\end{equation}
Moreover,
\begin{equation}
  \label{eq:stimalivellosharp}
  \jeps(u_\eps)\leq \gge(w_0)+C\eps.
\end{equation}
\end{lemma}
\begin{remark}
Throughout, the symbol $C$ will always denote (possibly different) constants that are
\emph{independent} of $\eps$ (but may depend on all the other
data, including the initial conditions $w_0, w_1$). 
\end{remark}
\begin{proof}
The function $\psi(t,x)=w_0(x)+\eps t w_1(x)$ satisfies the
boundary conditions \eqref{eq:condiniz}. We also have from
\eqref{growth2}, applied with $a=w_0$, $b=w_1$ and $\delta=\eps
t$, that
\[
\gge(w_0+\eps t w_1)\leq \gge(w_0)+C\eps t \left(1+\gge(w_0)+(\eps
t)^{\frac 1{1-\theta}}\right)
\]
having absorbed $\Vert w_1\Vert_W$ into $C$. Multiplying by $\emt$
and integrating, we find that
\[
\int_0^\infty \emt \gge\bigl(\psi(t)\bigr)\,dt \leq \gge(w_0) +
C\eps.
\]
Moreover, if $\kappa=1$, since $\psi'=\eps w_1$ and $w_1\in H$, from \eqref{eq:strutD} we see that
\[
\ksue \int_0^\infty \emt \diss(\psi'(t))\,dt= \frac {\eps}2
B(w_1,w_1)
\int_0^\infty \emt
\,dt \leq C\eps.
\]
Summing up, $\jeps(\psi)\leq \gge(w_0)+C\eps$: in particular,
$\jeps$ has a finite infimum and \eqref{eq:stimalivellosharp}
follows as soon as $\jeps$ has an absolute minimizer $u_\eps$.  To
show this, consider a minimizing sequence $u_k$ and fix $T>0$.
Combining the estimate
\[
\int_0^T \Vert u''_k(t)\Vert_{L^2}^2\,dt\leq e^T \int_0^T
\emt \Vert u''_k(t)\Vert_{L^2}^2\,dt\leq 2\eps ^2 e^T \jeps(u_k)
\]
with the initial conditions \eqref{eq:condiniz} satisfied by $u_k$,
we see that $\{u_k\}$ is bounded in $H^2_{\text{loc}}([0,\infty);L^2)$
whence, up to subsequences, $u_k(t)\rightharpoonup u(t)$
and $u'_k(t)\rightharpoonup u'(t)$ in $L^2$
for every $t\geq 0$, for some $u\in H^2_{\text{loc}}([0,\infty);L^2)$
that fulfills  \eqref{eq:condiniz}. Now the term involving $u''$ in
\eqref{eq:defjeps} is lower semicontinuous, and the same is
true of the other two terms by Fatou's Lemma and weak convergence
in $L^2$ of $u_k(t)$ and $u_k'(t)$ for fixed $t$, using \eqref{lsci} and
\eqref{lsci2}. This shows that $\jeps(u)\leq \liminf \jeps(u_k)$,
hence $u= u_\eps$ is a global minimizer.
\end{proof}

In some cases, a weaker version of \eqref{eq:stimalivellosharp} will be used, namely
\begin{equation}
\label{eq:stimalivello}
\jeps(u_\eps)\leq C.
\end{equation}

\begin{remark}
To simplify notation, given a minimizer $u_\eps$, we define, for $ t\ge 0$,
\begin{equation}
\label{short}
\Weps(t) := \gge\bigl(u_\eps(t, \cdot)\bigr)\quad\text{and}\quad \Heps(t):=\diss\bigl(u_\eps'(t, \cdot)\bigr).
\end{equation}
We also set
\begin{equation}
\label{eq:defD}
D_\eps(t):=\frac1{2\eps^2} \spint |u_\eps''(t,x)|^2\,dx\quad\text{for a.e. $t>0$,}
\end{equation}
so that we write
\begin{equation}
\label{eq:defL}
L_\eps(t):=D_\eps(t)+\Weps(t) + \ksue \Heps(t)
\end{equation}
for the locally integrable ``Lagrangian''.
Finally we introduce the kinetic energy function
\[
K_\eps(t):=\frac 1 {2\eps^2} \spint |u_\eps'(t,x)|^2\,dx\quad \forall t\geq 0.
\]
The notation just introduced will be used systematically in the sequel.
\end{remark}

Note that, due to Lemma \ref{lemmasommabilita} below,
 $K_\eps\in W^{1,1}(0,T)$ for all $T>0$ and
\begin{equation}
  \label{eq:defKprime}
  K_\eps'(t)=\frac 1 {\eps^2}\spint u_\eps'(t,x)u_\eps''(t,x)\,dx\quad\text{for a.e. $t>0$.}
\end{equation}

\begin{lemma}\label{lemmasommabilita}
The minimizers
$u_\eps$ defined by Lemma \ref{existence2} satisfy
\begin{align}
  \label{eq:5}
\int_0^\infty \emt D_\eps(t) \,dt &=  \int_0^\infty \emt\int
\frac{|u_\eps''|^2}{2\eps^2}\de\leq C,\\
  \label{eq:6}
 \int_0^\infty \emt K_\eps(t)\,dt &=   \int_0^\infty \emt\int
 \frac{|u_\eps'|^2}{2\eps^2}\de\leq
 C.
\end{align}
\end{lemma}

\begin{proof}
Estimate \eqref{eq:5} follows immediately from \eqref{eq:stimalivello}. The inequality (see \cite{Noi})
\[
\INTTO\infty \emt |v(t,x)|^2\,dxdt
\leq
2\spint |v(0,x)|^2\,dx+
4\INTTO\infty \emt |v'(t,x)|^2\,dxdt,
\]
applied with $v(t,x)=u_\eps'(t,x)$,  shows, using \eqref{eq:condiniz} and \eqref{eq:5}, that
\[
  \IT\infty \emt |u_\eps'|^2\de\leq 2\eps^2\spint| w_1(x)|^2\,dx+C\eps^2
\]
and \eqref{eq:6} is established since $w_1\in L^2$ by \eqref{densembed}.
\end{proof}

\section{The approximate energy}
Since integrals with an exponential weight play a major role in our investigation,
it is convenient to introduce the following \emph{average operator}.
\begin{definition}
If $f:\R^+\to [0,+\infty]$ is measurable, we let
\[
\ope f\,(s):=
\int_s^\infty e^{-(t-s)}f(t)\,dt,
\quad s\geq 0.
\]
\end{definition}
Note that $\ope f$ is well defined (possibly $+\infty$) as $f\geq 0$. However, since
\begin{equation}
\label{opeinzero}
\ope f\,(0)=\int_0^\infty \emt f(t)\,dt,
\end{equation}
if $\ope f\,(0)<\infty$ then
$\ope f$ is absolutely continuous on intervals $[0,T]$, and
\begin{equation}
\label{magia1}
(\ope f)'=\ope f-f.
\end{equation}
In any case, since $\ope f\geq 0$, starting from $f\geq 0$ one can iterate
$\ope$, and a simple computation gives
\begin{equation}
\label{opeopeinzero}
\ope^2\! f\,(s)=
\int_s^\infty e^{-(t-s)}(t-s)f(t)\,dt
\end{equation}
and, in particular,
\begin{equation}
\label{Aquadro}
\ope^2\! f\,(0)=
\int_0^\infty e^{-t}tf(t)\,dt.
\end{equation}

We now introduce a fundamental quantity for our approach.

\begin{definition}
Let $u_\eps$ be a minimizer of $\jeps$. The \emph{approximate energy} is the function
\begin{equation}
\label{defen}
\FF:=K_\eps+\ope^2 \Weps
\end{equation}
or, more explicitly,
\begin{equation}
\label{defFF} \FF(s)=K_\eps(s)+\int_s^\infty e^{-(t-s)}(t-s)
\gge(u_\eps(t))
\,dt,\quad s\geq 0.
\end{equation}
\end{definition}

\begin{remark}
In \eqref{defFF}, the kinetic energy $K_\eps$ is evaluated pointwise at
time $s$,
while the potential energy $\Weps$ is \emph{averaged}
over times $t\geq s$ via the \emph{probability kernel} $e^{-(t-s)}(t-s)$. 
However, recalling the time scaling $t\mapsto \eps t$
that links the functionals $\feps$ and $\jeps$,
in the original time scale the probability kernel in \eqref{defFF} concentrates
close to $s$ as $\eps\to 0$. Thus, heuristically, from \eqref{eq:uv} one expects that
$\FF(t/\eps)\approx {\mathcal E}(t)$ where $\mathcal E$ is the physical energy
defined in \eqref{eq:energiaonde}.
\end{remark}

Observe that, from \eqref{eq:defL} and \eqref{opeinzero}, we have
\[
\ope \Weps\,(0)\leq
\ope L_\eps\,(0)=\jeps(u_\eps)\leq C
\]
and so $\ope\Weps$ is well defined.
But since $\ope$ is iterated twice in \eqref{defen}, it is not even clear
why $\FF(s)$ should be finite. In fact, as we will show, $\FF(s)$ is finite
and \emph{decreasing}, and this monotonicity will be the key to our estimates.

The monotonicity of $\FF$ will be deduced from the following proposition.

\begin{proposition}
Let $u_\eps$ be a minimizer of $\jeps$. For every $g\in C^2([0,+\infty))$ such that $g(0) = 0$
and $g(t)$ is constant for large $t$, there results
\begin{equation}
\begin{split}
\label{energyident}
&\int_0^\infty \ems(g'(s)-g(s))L_\eps(s)\,ds
\\ &- \int_0^\infty \ems \bigl(4D_\eps(s)g'(s) +K_\eps'(s)g''(s)
+ \frac{2\kappa}{\eps}\Heps(s) g'(s)\bigr)\, ds
=
g'(0)R(u_\eps),
\end{split}
\end{equation}
where
\begin{equation}
\label{defRu}
R(u_\eps)=
- \eps \int_0^\infty \ems s \langle \nabla \gge(u_\eps(s)), w_1 \rangle\,ds
- \kappa \int_0^\infty \ems \langle \nabla \diss(u_\eps'(s)), w_1 \rangle\,ds.
\end{equation}
The quantity $R(u_\eps)$ is finite, and satisfies the estimate
\begin{equation}
\label{estRu}
|R(u_\eps)| \le C(\eps + \kappa \sqrt\eps) \le C\sqrt\eps.
\end{equation}
\end{proposition}

\begin{proof}
For every $\delta\in\R$ with $|\delta|$ small enough, the function
\begin{equation}
  \label{diffeo}
  \varphi(t)=\varphi(t,\delta)=t -\delta g(t)
\end{equation}
is a diffeomorphism of $\Rpiu$ of class $C^2$. We denote by $\psi$ its inverse,
\[
\psi(s)=\varphi^{-1}(s),\quad s\geq 0
\]
(the dependence on $\delta$, which is fixed, is omitted to
simplify the notation).
For small  $\delta$, we consider the competitor
\[
U(t)=u_\eps(\varphi(t)) + t\delta\eps g'(0) w_1,
\]
which satisfies the boundary conditions $U(0)=w_0$ and $U'(0)=\eps w_1$,
because $\varphi(0)=0$ and $\varphi'(0) = 1-\delta g'(0)$.  We have
\begin{align*}
&U'(t)=u_\eps'(\varphi(t))\varphi'(t)+ \delta\eps g'(0)w_1 ,\\
&U''(t)=u_\eps''(\varphi(t))|\varphi'(t)|^2+u_\eps'(\varphi(t))\varphi''(t)
\end{align*}
and hence
\begin{align*}
\jeps(U)&= \int_0^\infty \emt \left\{
\frac 1 {2\eps^2} \left\Vert
u_\eps''(\varphi(t))|\varphi'(t)|^2+u_\eps'(\varphi(t))\varphi''(t)
\right\Vert_{L^2}^2   \right.
\\
&+\gge\bigl(u_\eps(\varphi(t))  + t\delta\eps g'(0)w_1   \bigr)
+
\ksue
\diss\bigl(u_\eps'(\varphi(t))\varphi'(t)+ \delta\eps g'(0)w_1 \bigr)
 \Biggr\}\,dt.
\end{align*}
Changing variable in the integral letting $t=\psi(s)$, that is,
$s=\varphi(t)$, we have
\begin{equation}
\label{graffa}
\begin{split}
\jeps(U)&= \int_0^\infty \psi'(s)\emgs\left\{
 \frac 1 {2\eps^2} \left\Vert
u_\eps''(s)|\varphi'(\psi(s))|^2+u_\eps'(s)\varphi''( \psi(s))
\right\Vert_{L^2}^2
\right.
\\
+& \gge\bigl(u_\eps(s)  + \delta\eps g'(0)w_1 \psi(s)\bigr)
+
\ksue
\diss\bigl(u_\eps'(s)\varphi'(\psi(s)) +\delta\eps g'(0)w_1\bigr)
 \Biggr\}\,ds.
\end{split}
\end{equation}
 Note that, from \eqref{diffeo},
 $s=\varphi(\psi(s))=\psi(s)-\delta g(\psi(s))$, that is,
 \begin{equation}
   \label{eq:psi}
   \psi(s)=s+\delta g(\psi(s)).
 \end{equation}
In view of the assumptions on $g$, we have $\psi(s)\geq s-\delta\Vert
g\Vert_\infty$ and hence $ e^{-\psi(s)}\leq e^{\delta\Vert
g\Vert_\infty}\ems $. Furthemore, by \eqref{growth} and \eqref{eq:strutD},
\[
\gge\bigl(u_\eps(s)  + \delta\eps g'(0)w_1 \psi(s)\bigr) \le C\bigl(1 + \gge(u_\eps(s))+ \psi(s)^{\frac{1}{1-\theta}}\bigr)
\]
and
\[
\diss\bigl(u_\eps'(s)\varphi'(\psi(s)) +\delta\eps g'(0)w_1\bigr) \le 2 \varphi'(\psi(s))^2\diss(u_\eps'(s)) + C\diss(w_1).
\]
These inequalities, together with \eqref{eq:5},
\eqref{eq:6} and the finiteness of $\Vert\varphi'\Vert_\infty$ and
$\Vert\varphi''\Vert_\infty$, show that $\jeps(U)$ is finite and
hence $U$ is an admissible competitor.

Since $U(t)$ reduces to $u_\eps(t)$ when $\delta=0$, the
minimality of $u_\eps$ implies that
\begin{equation}
  \label{eq:derzero}
  \frac d{d\delta}\jeps(U)\bigl\vert_{\delta=0}\bigr. =0.
\end{equation}
In order to compute this derivative, we differentiate under the
integral sign in \eqref{graffa} (reasoning as above for the
finiteness of $\jeps(U)$, it is easy to prove that this is
possible). From \eqref{eq:psi},
\[
  \dedeltaz {\bigl(\psi'(s)\emgs\bigr)} =g(s)\ems-g'(s)\ems.
\]
Moreover, elementary computations give
\[
\dedeltaz {|\varphi'(\psi(s))|^2}=-2g'(s),\qquad
 \dedeltaz {\varphi''(\psi(s))}=-g''(s).
\]
Denoting by $\Theta(s)$ the function within braces under the
integral sign in \eqref{graffa}, and recalling \eqref{eq:defL}, there hold
\[
\Theta(s)\bigl\vert_{\delta=0}\bigr. =
\frac 1 {2\eps^2} \Vert u_\eps''(s)\Vert_{L^2}^2
+
\Weps(s)
+
\ksue\Heps(s)=L_\eps(s)
\]
and, recalling \eqref{eq:defD} and \eqref{eq:defKprime},
\begin{align*}
\dedeltaz{\Theta(s)} =  &-
 \frac 1 {\eps^2}
\left\langle u_\eps''(s), 2 u_\eps''(s)g'(s)+u_\eps'(s) g''(s)\right\rangle_{L^2}
-\frac {2\kappa}\eps g'(s)       \Heps(s)\\
&+ \eps g'(0) s\langle \nabla \gge(u_\eps(s)),w_1\rangle + \kappa g'(0)\langle\nabla \diss(u_\eps'(s)), w_1\rangle \\
=&-4D_\eps(s)g'(s)-K_\eps'(s)g''(s)-\frac {2\kappa}\eps g'(s)     \Heps(s)  \\
&+ \eps g'(0) s\langle \nabla \gge(u_\eps(s)),w_1\rangle + \kappa g'(0)\langle\nabla \diss(u_\eps'(s)), w_1\rangle.
\end{align*}
Combining these facts, we obtain that
\begin{align*}
\frac{\partial}{\partial \delta} &{\bigl( \psi'(s)\emgs \Theta(s)\bigr)}\bigl\vert_{\delta=0}\bigr.
= \ems \bigl(g'(s)-g(s)\bigr)
L_\eps(s)
\\
&-\ems
\Bigl( 4D_\eps(s)g'(s)+K_\eps'(s)g''(s)+\frac {2\kappa}\eps g'(s)       \Heps(s)
\Bigr)\\
& + \ems \bigl(\eps g'(0) s\langle \nabla \gge(u_\eps(s)),w_1\rangle + \kappa g'(0)\langle\nabla \diss(u_\eps'(s)), w_1\rangle\bigr).
\end{align*}
Finally, integrating in $s$ we see that \eqref{eq:derzero} reduces to \eqref{energyident}.

We now prove estimate \eqref{estRu}.
For the first integral in \eqref{defRu}, we have from \eqref{stimagattheta}
and Young inequality
\[
\begin{split}
&\left\vert
\int_0^\infty \ems s \langle \nabla \gge(u_\eps(s)), w_1 \rangle\,ds
\right\vert
\leq
\Vert w_1\Vert_W
\int_0^\infty \ems s \Vert \nabla \gge(u_\eps(s))\Vert_{W'}\,ds
\\
&\leq C
\int_0^\infty \ems s \bigl(1+ \Weps(s)^\theta\bigr)\,ds
=
C+
C
\int_0^\infty \ems s \Weps(s)^\theta\,ds
\\
&\leq C+
C\int_0^\infty \ems s^{1/(1-\theta)}\,ds+
\int_0^\infty \ems \Weps(s)\,ds
\leq C+\jeps(u_\eps)\leq C
\end{split}
\]
having used \eqref{eq:stimalivello}, and thus $|R(u_\eps)|\leq C\eps$ when $\kappa=0$.
If, on the other hand, $\kappa=1$, we also estimate the second integral in \eqref{defRu}:
\begin{align*}
&\left\vert
\int_0^\infty \ems \langle \nabla \diss(u_\eps'(s)), w_1 \rangle\,ds
\right\vert
\leq
\Vert w_1\Vert_H
\int_0^\infty \ems \Vert \nabla \diss(u_\eps'(s))\Vert_{H'}\,ds
\\
\leq&
C
\int_0^\infty \ems \sqrt{ \Heps(s)}\,ds
\leq
C\left(\int_0^\infty \ems \Heps(s)\,ds\right)^{1/2}
\leq C\left(\eps\jeps(u_\eps)\right)^{1/2}
\leq C\sqrt{\eps},
\end{align*}
having used \eqref{gradDDE}, Jensen inequality and \eqref{eq:stimalivello}.
\end{proof}

\begin{corollary}\label{prop42}
If $g\geq 0$ is of class $C^{1,1}$, satisfies $g(0)=0$ and is affine for large $t$, then
\eqref{energyident} remains true (all integrals being finite). In particular,
when $g(t)=t$, we obtain
\begin{equation}
\label{formulat}
\ope^2 L_\eps\,(0)
+\frac{2\kappa}\eps \ope \Heps\,(0)
+4\ope D_\eps\,(0)
= \ope L_\eps\,(0)
- R(u_\eps).
\end{equation}
\end{corollary}
\begin{remark}\label{remFfinita} Since $L_\eps(t)\geq \Weps(t)$, the finiteness of
$\ope^2L_\eps\,(0)$ in
\eqref{formulat} entails that the approximate energy $\FF(s)$ is finite for \emph{every}
$s\geq 0$ (in fact, it is absolutely continuous on intervals $[0,T]$, see the discussion after \eqref{opeinzero}).
\end{remark}

\begin{proof}
By smoothing a truncation of $g$,
one can find
 an increasing sequence $g_k$ of $C^2$ functions,
each eventually constant,
 such that as $k\to\infty$
\[
g_k(t) \uparrow g(t),\quad g_k'(t) \uparrow g'(t),\quad g_k''(t) \to g''(t)\qquad
\text{pointwise for every $t\geq 0$,}
\]
with $g'_k$ and $g''_k$ uniformly bounded.
We now write  \eqref{energyident} for $g_k$ and let $k \to \infty$. Since the functions
\[
\emt L_\eps(t), \quad \emt D_\eps(t),\quad \emt |K_\eps'(t)|,\quad \emt \Heps(t)
\]
are all in $L^1(\Rpiu)$ (either by the finiteness of $\jeps(u_\eps)$ or by Lemma \ref{lemmasommabilita}) and
$g_k'(0)R(u_\eps)$ does not depend on $k$,
all integrals pass to the limit, except for
the integral of $\emt g_k(t)L_\eps(t)$ because the $g_k$ are not uniformly bounded.
For this term, however, one can use \emph{monotone} convergence,
and the integral of
$\emt g(t)L_\eps(t)$ in the limit is finite, by finiteness of all other terms.
In particular,
one can let  $g(t)= t$ in \eqref{energyident}, which (recalling
\eqref{opeinzero} and \eqref{Aquadro}) yields \eqref{formulat}.
\end{proof}

\begin{corollary}
For almost every $T>0$ there results
\begin{equation}\label{formulatmenot}
\ope^2L_\eps\,(T)
-\ope L_\eps\,(T)
+K_\eps'(T)=-4\ope D_\eps\,(T)
-\frac{2\kappa}{\eps}\ope\Heps\,(T).
\end{equation}
\end{corollary}

\begin{proof}
Consider the function  $g \in C^{1,1}(\R)$ defined as
\[
g(t) = \begin{cases}
0 & \text{if}\quad t \le 0 \\
t^2/2 & \text{if }\quad t\in (0,1) \\
t-1/2 & \text{if}\quad t \ge 1
\end{cases}
\]
and, for   $T>0$ and $\delta >0$ (we will
let $\delta \downarrow 0$), set
\begin{equation}
\label{funzg}
g_\delta(t) = \delta g((t-T)/\delta).
\end{equation}
Each $g_\delta$ satisfies the assumptions of Corollary~\ref{prop42},
and $g_\delta''(t) = \frac1{\delta}\chi_{(T, T+\delta)}$.
Letting $g=g_\delta$ in \eqref{energyident} and
rearranging terms, gives
\begin{align*}
\label{energyident}
&\int_T^\infty \emt(g_\delta(t)- g'_\delta(t))L_\eps(t)\,dt
+ \frac{1}{\delta}\int_T^{T+\delta}\emt K_\eps'(t)\,dt \\
=&- \int_T^\infty \emt \bigl(4D_\eps(t) g_\delta'(t)
+\frac{2\kappa}{\eps}\Heps(t)g_\delta'(t)\bigr)\, dt.
\end{align*}
Note that, as $\delta \to 0$, $g_\delta(t) \to (t-T)^+$ while $g_\delta'(t) \to \chi_{(T,\infty)}$,
with bounds
$|g_\delta(t)|\leq  (t-T)^+$ and $|g'_\delta(t)|\leq  1$.
By dominated convergence we can let $\delta \downarrow 0$, thus obtaining for a.e. $T$
\begin{align*}
&\int_T^\infty\emt  (t-T)L_\eps(t)\,dt
-\int_T^\infty \emt L_\eps(t)\,dt
+ e^{-T} K_\eps'(T)\\
=&-\int_T^\infty \emt \bigl(4D_\eps(t) +\frac{2\kappa}\eps \Heps(t)\bigr)\,dt,
\end{align*}
and multiplying by $e^T$ one
obtains \eqref{formulatmenot}.
\end{proof}

\begin{theorem}
The function $\FF$ is finite and decreasing. More precisely,
\begin{equation}
\label{FFprimo} \FF'(T)\leq
-\,\frac{\,\kappa\,}{\eps}\left(
\ope
\Heps\,(T)
+
\ope^2
\Heps\,(T)\right),
\end{equation}
and
\begin{equation}
\label{stimaFT}
\FF(T)+\frac{2\kappa}{\eps}\int_0^T
\Heps(t)\,dt
\leq
\frac 1 2 \Vert w_1\Vert_{L^2}^2+\gge(w_0)+C\eps
+C\kappa\sqrt{\eps}
,\quad\forall T\geq 0.
\end{equation}
\end{theorem}

\begin{proof}
From Remark \ref{remFfinita} we know that
$\FF$ is absolutely continuous on intervals $[0,T]$.
Hence, differentiating \eqref{defen} and using \eqref{magia1} written with $f=\ope \Weps$ yields
\[
\FF'=K_\eps'-\ope \Weps+\ope^2 \Weps.
\]
But since $\Weps=L_\eps-D_\eps-\frac{\kappa}{\eps}\Heps$,
using \eqref{formulatmenot} we obtain
\[
\FF'=-3\ope D_\eps-\ope^2 D_\eps-\frac{\kappa}{\eps}\ope
\Heps
-\frac{\kappa}{\eps}\ope^2
\Heps,
\]
and \eqref{FFprimo} follows. Choose now
$f=\frac\kappa\eps\Heps$, so that
\eqref{FFprimo} reads $\FF'+\ope f+\ope^2 f\leq 0$. Integrating we
find
\begin{equation}
\label{lll}
\FF(T)+\int_0^T \ope f\,dt
+\int_0^T \ope^2 f\,dt\leq \FF(0).
\end{equation}
For the former integral, using \eqref{magia1} we have
\[
\int_0^T \ope f\,dt
=
\int_0^T f\,dt+\ope f\,(T)-\ope f\,(0).
\]
For the latter, iterating twice the same argument gives
\[
\int_0^T \ope^2 f\,dt
=
\int_0^T f\,dt+\ope^2 f\,(T)+\ope f\,(T)-\ope^2 f\,(0)-\ope f\,(0),
\]
so that \eqref{lll}, in particular, yields
\begin{align*}
&\FF(T)+2\int_0^T f(t)\,dt\leq
\FF(0)+\ope^2 f\,(0)+2\ope f\,(0)=\\
&K_\eps(0)+\ope^2\Weps\,(0)+\ope^2 f\,(0)+2\ope f\,(0)
\leq
\frac 1 2\Vert w_1\Vert_{L^2}^2 +\ope^2 L_\eps\,(0)+\frac{2\kappa}{\eps}\ope H\,(0).
\end{align*}
Therefore, since $4D_\eps(t)\geq 0$, using
\eqref{formulat}
we find that
\[
\FF(T)+2\int_0^T f(t)\,dt\leq
\frac 1 2\Vert w_1\Vert_{L^2}^2
+\ope L_\eps\,(0)
 -R(u_\eps),
\]
and since $\ope L_\eps\,(0)=\jeps(u_\eps)$, from \eqref{eq:stimalivellosharp} we see that
\eqref{stimaFT} follows from \eqref{estRu}.
\end{proof}

\section{Proof of the a priori estimates}
In this section
we prove part $(a)$ of theorems \ref{mainteo1} and \ref{mainteo2}.

As discussed
at the beginning of Section \ref{sec2}, the minimizers $\weps$
of $\feps$ in \eqref{eq:deffepsmisto} (subject to \eqref{eq:condin2})
are related to the minimizers $u_\eps$
of $\jeps$ in \eqref{eq:defjeps} (subject to \eqref{eq:condiniz}) by the
change of variable \eqref{eq:uv}
and in particular the functions $w_\eps$ satisfy the boundary conditions
\begin{equation}
  \label{eq:icv}
  \weps(0,x)=w_0(x),\quad \weps'(0,x)= w_1(x).
\end{equation}
So the estimates on $\weps$ will follow from analogous estimates on $u_\eps$ by scaling.

\begin{proof}[Proof of \eqref{eq:stimaaprioriprime}]
Scaling as in \eqref{eq:uv} and using \eqref{stimaFT} and \eqref{defFF} yields
\[
\frac 1 2\spint |\weps'(t,x)|^2\,dx=
K_\eps(t/\eps) \leq C,
\]
which proves the first estimate
in \eqref{eq:stimaaprioriprime}. The second estimate follows
immediately from the first and the boundary condition  in \eqref{eq:icv}, since $w_0 \in W\embed L^2$.
\end{proof}

\begin{proof}[Proof of \eqref{provvD}] When $\kappa=1$, observe that \eqref{stimaFT} gives
\begin{equation}
\label{serveanchedopo}
\int_0^\infty \Heps(t)\,dt
=\int_0^\infty \diss(u'_\eps(t))\,dt\leq C\eps,
\end{equation}
and \eqref{provvD} follows from \eqref{eq:uv} and scaling, using \eqref{eq:strutD}.
\end{proof}

\begin{proof}[Proof of \eqref{eq:stimaapriori}]
Since $L_\eps\geq 0$, we have from \eqref{eq:stimalivello}
\begin{equation}
\label{tpiccolo}
e^{-2}\int_0^2 \Weps(t)\,dt\leq \int_0^2 \emt L_\eps(t)\,dt
\leq  \jeps(u_\eps)\leq C.
\end{equation}
In the same spirit, we have for every $s\geq 0$,
\[
e^{-2}\int_{s+1}^{s+2} \Weps(t)\,dt
\leq  \int_{s+1}^{s+2} (t-s)e^{-(t-s)} \Weps(t)\,dt
\leq  \ope^2 \Weps\,(s)\leq  \FF(s)\le C.
\]
which, combined with \eqref{tpiccolo}, yields
\begin{equation}
\label{stimass1}
\int_s^{s+1} \Weps(t)\,dt\leq C\quad\forall s\geq 0.
\end{equation}
Writing $s=\tau/\eps$ and scaling, recalling \eqref{short} we obtain that
\begin{equation}
\label{cover}
\int_\tau^{\tau+\eps} \gge \bigl(w_\eps(z)\bigr)\,dz\leq C\eps\quad\forall \tau\geq 0.
\end{equation}
Now, if $\tau\geq 0$ and $T\geq\eps$ as in \eqref{eq:stimaapriori},
by covering $[\tau,\tau+T]$ with consecutive intervals of length $\eps$
and using \eqref{cover} in each interval, one obtains \eqref{eq:stimaapriori}.
\end{proof}

In the next lemma we are going to use the inequality
\begin{equation}
\label{stimagg} \int_t^{t+1} \Vert \nabla \gge
\bigl(u_\eps(t)\bigr)\Vert_{W'}^{\frac 1 \theta}\,dt
\leq C\quad\forall t\geq 0,
\end{equation}
which follows immediately on combining \eqref{stimagattheta} and
\eqref{stimass1}.

\begin{lemma}[Euler--Lagrange equation]\label{lemmaeul}
Suppose that $\eta(t,x)=\varphi(t)h(x)$, with  $\varphi \in C^{1,1}([0,+\infty))$,
$\varphi(0)=\varphi'(0)=0$ and $h\in W\cap H$.
Then
\begin{equation}
  \label{eq:eulerp}
\int_0^\infty
\emt
\left(
\frac 1{\eps^2}
\langle
u''_\eps,\eta''
\rangle_{L^2}
+
\langle \nabla
\gge\bigl(u_\eps(t)\bigr), \eta\rangle + \ksue\langle \nabla
\diss\bigl(u'_\eps(t)\bigr), \eta'\rangle \right)\,dt =0.
\end{equation}
Moreover, the same conclusion holds if $\eta\in C^\infty_0(\R^+\times \R^n)$.
\end{lemma}
\begin{proof}
The Euler--Lagrange equation \eqref{eq:eulerp} corresponds
to the condition $f'(0)=0$ where $f(\delta)=\jeps(u_\eps+\delta\eta)$; it is enough 
to justify differentiation under the
integral sign in \eqref{eq:defjeps} in the term involving $\gge$ (the
term with $\diss$ is quadratic due to \eqref{eq:strutD}).

First consider the case where $\eta=\varphi(t)h(x)$, and set $v=u_\eps+\delta \eta$ with,
say, $|\delta|\leq 1$. As $\varphi\in C^{1,1}$, $\varphi(t)$ grows at most quadratically as $t\to\infty$;
applying \eqref{growth2} with
$a=u_\eps(t)$ and $b=\varphi(t) h$, multiplying by $\emt$ and integrating,
one sees that $\jeps(v)$ is finite (and $v$ satisfies the boundary conditions
\eqref{eq:condiniz}). For a.e. $t>0$, we have
\[
\frac {d}{d\delta} \gge\bigl(u_\eps(t)+\delta \eta(t)\bigr)
_{\big\vert\delta=0}
=
\langle
\nabla \gge\bigl(u_\eps(t)\bigr),\eta(t)\rangle=
\varphi(t)
\langle
\nabla \gge\bigl(u_\eps(t)\bigr),h\rangle
\]
and this function, multiplied by $\emt$, is integrable on $\R^+$
due to \eqref{stimagg}.
Indeed, one can easily check that differentiation
in $\delta$ under the integral sign
is justified,
now using \eqref{stimalip},
with $a$ and $b$ as before.

Now consider a generic test function $\eta\in C^\infty_0(\R^+\times \R^n)$.
Due to \eqref{stimagg} and \eqref{densembed}, the left hand side
of \eqref{eq:eulerp} defines
 a distribution on $\R^+\times \R^n$.
If $\eta=\varphi(t)h(x)$ with $\varphi\in C^\infty_0(\R^+)$ and
$h\in C^\infty_0(\R^n)$, then in particular $\varphi\in C^{1,1}([0,+\infty))$ and
\eqref{eq:eulerp} has just been established. The general case then follows
from the fact that  test function of the form
 $\varphi(t)h(x)$ are dense in $C^\infty_0(\R^+\times \R^n)$ (see \cite{Schwartz},
Chap.~IV, and in particular Thm.~III).
\end{proof}

\begin{proof}[Proof of \eqref{eq:stimaaprioridue}, \eqref{eq:stimaaprioriduebis}.] These
estimates will follow from the following representation formula (proved below)
for $u''_\eps$, valid for a.e. $T>0$:
\begin{equation}
\label{reprformula}
\frac 1{\eps^2}
\langle
u''_\eps(T),
h\rangle_{L^2}
=
-\ope^2 f_1\,(T)
-\frac \kappa\eps
\ope f_2\,(T),\qquad
h\in\begin{cases}
W &\text{if $\kappa=0$},\\
W\cap H &\text{if $\kappa=1$}
\end{cases}
\end{equation}
where
\begin{equation}
\label{defunodue}
f_1(t)=\langle \nabla\gge(u_\eps(t)),h\rangle,\qquad
f_1(t)=\langle \nabla\diss(u'_\eps(t)),h\rangle.
\end{equation}
Note that using \eqref{stimagat},
\[
|f_1(t)|\leq \Vert h\Vert_W
\Vert
\nabla\gge(u_\eps(t))
\Vert_{W'}
\leq
C\Vert h\Vert_W (1+\Weps(t)).
\]
But since $\ope^2 1=1$ by \eqref{opeopeinzero}, and $\ope^2 \Weps\leq \FF\leq C$
by \eqref{defen} and \eqref{stimaFT}, we have
\begin{equation}
\label{stimaf1}
\left\vert
\ope^2 f_1\,(T)\right\vert
\leq
\ope^2 |f_1|\,(T)\leq
C\Vert h\Vert_W\quad
\forall T\geq 0.
\end{equation}
Thus, if $\kappa=0$, \eqref{reprformula} can be seen (via the second inclusion in
\eqref{densembed}) as a representation formula for $u''_\eps(T)$ as an element
of $W'$, and the last estimate gives
\[
\frac1{\eps^2}\Vert u''_\eps(T)\Vert_{W'}\leq C\quad
\text{for a.e.  $T\geq 0$.}
\]
Then, scaling according to \eqref{eq:uv}, one obtains \eqref{eq:stimaaprioridue}.

In addition, if $\kappa=1$ and $h\in W\cap H$, we have using \eqref{gradDDE}
\[
|f_2(t)|\leq \Vert h\Vert_H
\Vert
\nabla\diss(u'_\eps(t))
\Vert_{H'}
\leq
C\Vert h\Vert_H \sqrt{\Heps(t)}
\]
and thus, using \eqref{serveanchedopo},
$\Vert f_2\Vert_{L^2(\R^+)}\leq C\sqrt{\eps}\Vert h\Vert_{H}$.
Therefore, since the operator
$\ope$ maps  $L^2(\Rpiu)$ continuously into itself, we find that
\[
\Vert \ope f_2\Vert_{L^2(\Rpiu)}\leq C\sqrt{\eps}\Vert h\Vert_{H}.
\]
Then, recalling \eqref{HembedL2}, \eqref{reprformula} can be written
as $u''_\eps/\eps^2=\Phi_1+\Phi_2$,
with the bounds $\Vert \Phi_1\Vert_{L^\infty(\Rpiu;W')}\leq C$ by \eqref{stimaf1},
and $\Vert \Phi_2\Vert_{L^2(\Rpiu;H')}\leq C/\sqrt{\eps}$ by the previous
inequality. Scaling according to \eqref{eq:uv}, this means that
$w''_\eps(t)=\Phi_1(t/\eps)+\Phi_2(t/\eps)$, and \eqref{eq:stimaaprioriduebis}
follows since $\Vert \Phi_2(t/\eps)\Vert_{L^2}=\sqrt{\eps}\Vert\Phi_2\Vert_{L^2}$.

It remains to prove \eqref{reprformula}.
For $T,\delta>0$, we take the $C^{1,1}$ function $g_\delta$ defined in \eqref{funzg}.
Given $h$ as in \eqref{reprformula}, we set $\eta(t,x)=g_\delta(t) h(x)$ and we apply Lemma~\ref{lemmaeul}.

As $g_\delta''(t)=
\delta^{-1} \chi_{(T,T+\delta)}(t)$, \eqref{eq:eulerp} multiplied by $e^T$ reads
\begin{align*}
\frac {e^T} {\eps^2 \delta}\int_T^{T+\delta} e^{-t}
\langle u_\eps''(t), h\rangle_{L^2} \,dt = -
\int_T^\infty e^{-(t-T)}
\left(
g_\delta(t)
f_1(t)
+
\ksue
g'_\delta(t)
f_2(t)
\right)
\,dt
\end{align*}
with $f_1,f_2$ as in \eqref{defunodue}.
Since $|g_\delta(t)|\leq  (t-T)^+$ and $|g_\delta'(t)|\leq \chi_{(T,\infty)}$,
one can dominate the integrand functions as done above for
$f_1$ and $f_2$.
Finally, letting $\delta\downarrow 0$, $g_\delta \to (t-T)^+$ and $g'_\delta \to \chi_{(T,\infty)}$,
and one obtains \eqref{reprformula} for a.e. $T$.
\end{proof}

\section{Proof of convergence and energy inequality}
In this section we first prove parts (b) and (c) of theorems \ref{mainteo1}
and \ref{mainteo2}. Then, we prove Theorem~\ref{mainteo3}.

In the sequel, we deal with a sequence of minimizers
$w_{\eps_i}$ as in $(b)$ of Theorem~\ref{mainteo1}, and
we will \emph{tacitly} extract  several subsequences. For ease of notation, however, we will denote by
$w_\eps$ the original sequence, as well as the subsequences we extract.

\begin{proof}[Proof of part (b): passage to the limit.] Regardless of $\kappa\in\{0,1\}$,  \eqref{eq:stimaaprioriprime} shows
that the $\weps$ are equibounded in $H^1\loc([0,\infty); L^2)$.
Precisely, for every $T>0$ there exists a constant $C_T$ such that
\begin{equation}
\label{bounds}
\|\weps \|_{H^1((0,T); L^2)}^2 = \int_0^T \left( \|\weps'(t)\|_{L^2} + \|\weps(t)\|_{L^2}\right)\, dt \le C_T.
\end{equation}
Thus there exists  a function $w \in H^1\loc([0,\infty); L^2)$ such that
\begin{equation}
\label{limitv}
\weps \rightharpoonup w\quad \text{in $H^1\loc([0,\infty); L^2)$}\quad\text{and}\quad
\weps(t) \rightharpoonup w(t)\quad\text{ in $L^2$}\;\;
\forall t\geq 0
\end{equation}
as $\eps \to 0$. Clearly, the claims on $w'$, $w''$ in \eqref{eq:welleinfinito1}
and \eqref{wsecondobis}
 follow
from the uniform bounds in \eqref{eq:stimaaprioriprime}, \eqref{eq:stimaaprioridue}
and \eqref{eq:stimaaprioriduebis}.
Moreover, when $\kappa = 1$,
since $H$ is normed by
$\Vert v\Vert_H^2=\Vert v\Vert_{L^2}^2+2\diss(v)$,
\eqref{provvD} combined with
\eqref{bounds}
provide a uniform bound for $\weps'$
in $L^2((0,T); H)$ for every $T>0$, whence \eqref{convL2H}.

To prove
that $w$ satisfies \eqref{eq:condin2}, we recall that
these two conditions are satisfied, by assumption,
by each $w_\eps$:
then the first condition for $w$ follows easily from
the second part of
\eqref{limitv}, considering $t=0$.

For the second condition,
if
$\kappa=0$ then \eqref{eq:stimaaprioridue} and \eqref{eq:stimaaprioriprime}
(combined with $L^2\embed W'$, that follows from \eqref{densembed}) yield
a uniform bound for $w_\eps'$ in $W^{1,\infty}(\R^+;W')$, which guarantees
the maintenance, in the limit, of $w_\eps'(0)=w_1$ (now viewed as an equality in $W'$).
If $\kappa=1$ then the argument is similar: since $W\cap H\embed L^2$ densely
by \eqref{densembed} and \eqref{HembedL2},
in particular \eqref{eq:stimaaprioriduebis} yields a uniform bound for
$w_\eps''$ in $L^2((0,1);(W\cap H)')$, hence a bound for $w_\eps'$ in
$H^1((0,1);(W\cap H)')$, sufficient to guarantee that $w'(0)=w_1$
(now seen as an equality in $(W\cap H)'$).
\end{proof}

\begin{proof}[Proof of part (c): energy inequality.]
To obtain \eqref{eq:quasiconservation}
and \eqref{eq:quasiconservation2} we need the following Lemma, proved in
\cite{Noi} and reformulated here in terms of the operator $\ope$.

\begin{lemma}\label{lemmatech}
Let $l(t)$, $m(t)$ be nonnegative functions in $ L^1_\text{loc}$, such that
\begin{equation}
\label{ipotesi}
(\ope^2 l)(t)\leq m(t)\quad\text{ for
a.e. $t>0$.}
\end{equation}
Then, for every pair of numbers
$a>0$ and $\delta\in (0,1)$,
\[
\left(\int_0^{\delta a}s e^{-s}\,ds\right)\int_{T+\delta a}^{T+a}l(t)\,dt
\leq \int_T^{T+a} m(t)\,dt\quad\forall T\geq 0.
\]
\end{lemma}

Recalling \eqref{defen} and \eqref{stimaFT},
we can apply  Lemma \ref{lemmatech} with $l(t)= \Weps(t)$ and
\[
m(t)=
-K_\eps(t)-\frac{2\kappa}{\eps}\int_0^t
\Heps(s)\,ds
+
\frac 1 2 \Vert w_1\Vert_{L^2}^2+\gge(w_0)+C\sqrt{\eps}
\]
(assumption \eqref{ipotesi} corresponds to \eqref{stimaFT} via \eqref{defen}). This gives,
 for every $T\ge 0$, every $a>0$ and every $\delta \in (0,1)$,
\[
\begin{split}
Y(\delta a) &\int_{T+\delta a}^{T+a} \Weps(t)\,dt  \\
&\leq
-\int_{T}^{T+a}\left(K_\eps(t)+\frac{2\kappa}{\eps}\int_0^t \Heps(s)\,ds \right)\,dt
+a{\mathcal E}(0)+aC\sqrt{\eps}
\end{split}
\]
where, for simplicity, $Y(z)=\int_0^{z} s e^{-s}\,ds$
and ${\mathcal E}(0)$ is defined as in \eqref{eq:quasiconservation}.
Now, recalling \eqref{eq:uv}, we want to rewrite this estimate in terms of $w_\eps$
instead of $u_\eps$: in view of this, it is convenient to first replace $T$ with
$T/\eps$ and $a$ with $a/\eps$, and then change variable in the integrals
according to \eqref{eq:uv}, thus obtaining, rearranging terms,
\[
\begin{split}
Y(\delta a/\eps) &\int_{T+\delta a}^{T+a} \gge(w_\eps(t))\,dt
+\int_{T}^{T+a}\left(\frac 1 2 \Vert w_\eps'(t)\Vert_{L^2}^2
+2\kappa\int_0^t \diss (w_\eps'(s))\,ds \right)\,dt
\\
&\leq a{\mathcal E}(0)+aC\sqrt{\eps},\qquad
\forall T\geq 0,\quad\forall a>0,\quad\forall\delta\in (0,1).
\end{split}
\]
Now, for fixed $T,a,\delta$, recalling \eqref{limitv} we let $\eps\to 0$
in the previous estimate. Since $Y(\delta a/\eps)\to 1$, we obtain
by semicontinuity
\[
\int_{T+\delta a}^{T+a} \gge(w(t))\,dt
+\int_{T}^{T+a}\left(\frac 1 2 \Vert w'(t)\Vert_{L^2}^2
+2\kappa\int_0^t \diss (w'(s))\,ds \right)\,dt
\leq a{\mathcal E}(0)
\]
(for the integral involving $\gge$ one uses Fatou's Lemma,
\eqref{limitv} and \eqref{lsci}, while the double integral with $\diss$
is a convex and strongly continuous function of $w'_\eps$ in
$L^2((0,T+a);H)$, and one may use \eqref{convL2H}).
Now we let $\delta\to 0^+$ (with $T$ and $a$ fixed), then we divide by $a$
and finally we let $a\to 0^+$, to obtain
\[
\gge(w(T))+\frac 1 2 \Vert w'(T)\Vert_{L^2}^2
+2\kappa\int_0^T \diss (w'(s))\,ds
\leq {\mathcal E}(0)\quad\text{for a.e. $T\geq 0$.}
\]
When $\kappa=0$ this reduces to \eqref{eq:quasiconservation}, while
when $\kappa=1$ this reduces to \eqref{eq:quasiconservation2}.
\end{proof}

\begin{lemma}
For every test function $\eta\in C^\infty_0(\spt)$, there holds
\begin{equation}
\begin{split}
\label{eulerp2}
 &\int_0^\infty
 \langle w_\eps'(\tau), \eps^2\eta'''(\tau)+2\eps\eta''(\tau)+\eta'(\tau)\rangle_{L^2}
 \,d\tau\\
=
 &\int_0^\infty
 \left(
\langle \nabla
\gge\bigl(w_\eps(\tau)\bigr), \eta(\tau)\rangle
+
\kappa\langle \nabla \diss\bigl(w'_\eps(\tau)\bigr), \eta(\tau)+\eps\eta'(\tau)\rangle
\right)\,d\tau.
\end{split}
\end{equation}
\end{lemma}

\begin{proof}
Let $\psi\in C^\infty_0(\spt)$. Choosing
$\eta=e^t\psi$ in \eqref{eq:eulerp} gives
\begin{equation*}
\begin{split}
 &\int_0^\infty
\frac 1 {\eps^2}
\langle u_\eps''(t), \psi''(t)+2\psi'(t)+\psi(t) \rangle_{L^2}\,dt \\
=&-\int_0^\infty \left(
\langle \nabla
\gge\bigl(u_\eps(t)\bigr), \psi(t)\rangle
+
\ksue\langle \nabla\diss\bigl(u'_\eps(t)\bigr), \psi(t)+\psi'(t)\rangle
\right)\,dt .
\end{split}
\end{equation*}
Now we replace $u_\eps$ with $w_\eps$ using \eqref{eq:uv} and, accordingly,
we take $\psi$ of the form $\psi(t,x)=\eta(\eps t,x)$, for an arbitrary test function
$\eta$. Plugging into the last equation and changing variable $\tau=\eps t$
in each integral, one obtains \eqref{eulerp2} after
integrating by parts the first term.
\end{proof}

\begin{proof}[Proof of Theorem \ref{mainteo3}] The functional $v\mapsto 1/2\Vert v\Vert_{\dot H^m}^2$,
clearly satisfies assumption (H1), on letting $W=H^m$ and $\theta=1/2$:
then first part of the claim
follows on combining Remarks~\ref{rema1} and~\ref{rema2}. Thus, one may apply
Theorem~\ref{mainteo1} (or~\ref{mainteo2}, if $\kappa=1$). We wish to
pass to the limit in \eqref{eulerp2}, in particular in the nonlinear
term involving $\nabla\gge$, namely we wish to prove that (up to subsequences)
\begin{equation}
\label{nllimit}
\lim_{\eps\to 0}
\int_0^\infty  \langle \nabla
\gge\bigl(w_\eps(\tau)\bigr), \eta(\tau)\rangle \,d\tau
=
\int_0^\infty  \langle \nabla
\gge\bigl(w(\tau)\bigr), \eta(\tau)\rangle \,d\tau,
\end{equation}
where $w$ is the limit function obtained by Theorem~\ref{mainteo1} (or \ref{mainteo2}).
Due to \eqref{struct2} (see also Remark~\ref{rema1} and \eqref{gradiente}),
we have
\[
\begin{split}
&\int_0^\infty  \langle \nabla
\gge\bigl(w_\eps(\tau)\bigr), \eta(\tau)\rangle \,d\tau
=
\int_0^\infty \langle w_\eps(\tau),\eta(\tau)\rangle_{\dot H^m}\,d\tau\\
+&
\sum_{0\leq k<m}
\lambda_k \int_0^\infty
\int \left| \nabla^k w_\eps(\tau) \right|^{p_k-2}\nabla^k w_\eps(\tau)\cdot
\nabla^k\eta(\tau)\,dxd\tau.
\end{split}
\]
Thus, to prove \eqref{nllimit}, we need \emph{strong} convergence of
$|\nabla^k \weps|^{p_k-2}\nabla^k \weps$ ($k<m$) in $L^{1}(Q)$, for every
cylinder $Q=(0,T)\!\times\! B$ ($B$ being a ball in $\R^n$),
and \emph{weak} convergence of $w_\eps$ in $L^2((0,T);H^m)$.

Now, due to \eqref{struct2}, the bounds in \eqref{eq:stimaapriori} (with $\tau=0$)
take the concrete form
\begin{equation*}
\int_0^T \left(
\Vert w_\eps(t)\Vert_{\dot H^m}^2+
\sum_{0\leq k<m}\lambda_k
\int_{\R^n}
\left| \nabla^k \weps(t,x)\right|^{p_k}
\,dx\right)\,dt\leq C_T\quad (T\geq 1),
\end{equation*}
so that, combining with the second part of \eqref{eq:stimaaprioriprime},
$w_\eps$ is \emph{weakly compact} in
$L^2((0,T);H^m)$, while
$\nabla^k \weps$  is \emph{bounded} in $L^{p_k}(Q)$
(we focus on  those $k$ for which $\lambda_k>0$).
Thus, to conclude,
it suffices  to check
the \emph{strong} convergence of $\nabla^k \weps$
in $L^2(Q)$  (this condition is even stronger than necessary,
if $p_k<3$). Now
fix a cylinder $Q=(0,T)\!\times\! B$. If $0\leq k<m$ we have that
$H^m(B)\embed H^{k}(B)\embed L^2(B)$,
and the first injection is compact: thus,
combining
the bound for $w_\eps$
in $L^2((0,T);H^m)$ with the bound
for $w_\eps'$ in
$L^2((0,T);L^2(B))$ (from \ref{eq:stimaaprioriprime}),
we obtain the \emph{strong compactness} of $w_\eps$
in $L^2((0,T);H^k(B))$
(see e.g. Thm.~5.1 in \cite{Lions}), whence $\nabla^k w_\eps$ converges strongly in
$L^2(Q)$.

The terms in \eqref{eulerp2} other than $\nabla\gge(w_\eps)$
are \emph{linear} in $w_\eps$, and using (b) of Theorem~\ref{mainteo1}
one can pass to the limit in \eqref{eulerp2} when $\kappa=0$.
Finally, if $\kappa=1$, recalling \eqref{eq:strutD} and \eqref{gradDDE}, also the term involving
$\nabla\diss(w_\eps)$ passes to the limit in \eqref{eulerp2}, using
\eqref{convL2H} and \eqref{gradDDE}. In either case, taking the limit in \eqref{eulerp2} one obtains
\begin{equation*}
\int_0^\infty
 \langle w'(\tau), \eta'(\tau)\rangle_{L^2}
 \,d\tau\\
=
 \int_0^\infty
 \Bigl(
\langle \nabla
\gge\bigl(w(\tau)\bigr), \eta(\tau)\rangle
+
\kappa\langle \nabla \diss\bigl(w'(\tau)\bigr), \eta(\tau)\rangle
\Bigr)\,d\tau
\end{equation*}
that is, $w$ is a weak solution of \eqref{eq::1} (or \eqref{eq::2}, if $\kappa=1$).
\end{proof}

\section{Examples}\label{esempi}
In this section we show how several concrete problems fit into the
general scheme
described above. Let us begin with equations without dissipative terms.

\begin{list}{\stepcounter{enumi}\emph{\arabic{enumi}.}\ }
{ \setlength{\labelwidth}{0cm} \setlength{\labelsep}{0cm}
\setlength{\leftmargin}{0cm} }

\item \emph{Linear equations.} These are obtained when $\gge$ is a
quadratic functional, e.g.
\[
\gge (v) = \frac12 \sum_{j\in {\cal R}} \int |\partial^j v |^2
\,dx,
\]
where ${\cal R}\subset {\mathbb N}^n$ is a finite set of
multiindices and $\partial^j$ denotes partial differentiation. In
this case the natural choice for the domain of  $\gge$ is
$W=\{v\in L^2\;|\; \partial^j v\in L^2,\; \forall j\in{\cal
R}\}$, and assumption (H1) is fulfilled (in particular,
\eqref{stimagattheta} is satisfied with $\theta=1/2$).

Reasoning as in Remark~\ref{rema1}, the hyperbolic equation corresponding to
\eqref{eq::1} is
\[
 w'' =- \sum_{j \in {\cal R}}(-1)^{|j|} \partial^{2j} w.
\]
In this case Theorem \ref{mainteo3} applies ($\kappa=0$), and
Problem~1 can be completely solved. Concrete instances are the
linear wave equation $w''=\Delta w$, the Klein--Gordon equation
$w''=\Delta w-w$, or the bi--harmonic wave equation $w''=-\Delta^2
w$.

\item \namedlabel{exNLW}{\arabic{enumi}} \emph{Defocusing NLW.}
This matches De~Giorgi's original conjecture in \cite{DG}, and has
been dealt with in \cite{Noi}. It corresponds to the choice
\[
\gge (v) =\int\left(\frac 1 2   |\nabla v |^2 + \frac1p
|v|^p\right) \,dx
\]
in \eqref{eq:deffeps}, for some $p>2$. Here, by Remark~\ref{rema1}, \eqref{eq::1} takes
the concrete form
\[
w'' = \Delta w -|w|^{p-2}w = 0.
\]
Letting $W= H^1 \cap L^p$, assumption (H1) is satisfied (with
$\theta=1-1/p$ in \eqref{stimagattheta}), and all the results in
\cite{Noi} are recovered as an application of  Theorem
\ref{mainteo3}.

\item\namedlabel{exLastND}{\arabic{enumi}}
\emph{Sine--Gordon equation.} If we let
\[
W(v)=\int\left( \frac 12 |\nabla v|^2+1-\cos v\right)\,dx
\]
with domain $W=H^1$, then \eqref{eq::1} becomes the sine--Gordon
equation
\[
w''=\Delta w-\sin w.
\]
Then (H1) is fulfilled with $\theta=1/2$, and Theorem~\ref{mainteo3} applies. Note that the functional $\gge$
associated with this problem is not convex.

\item \namedlabel{exQuasilin}{\arabic{enumi}} \emph{Quasilinear
wave equations.} Powers other than 2 on the gradient term in
$\gge$ give rise to quasilinear wave equations. For example
\[
\gge (v) =  \frac1p \int |\nabla v |^p \,dx,\quad\text{or}\quad
\gge (v) = \int\left( \frac1p  |\nabla v |^p + \frac1q
|v|^q\right) \,dx \qquad(p,q>1)
\]
correspond, respectively, to the quasilinear wave equations
\[
w'' = \Delta_p w  \qquad\text{and}\qquad  w'' = \Delta_p w -
|w|^{q-2}w,
\]
where $\Delta_p$ is the $p$--laplacian. Assumption (H1) is
satisfied, for the former equation, letting  $W= \{v\in L^2\;|\;
\nabla v\in L^p\}$ and $\theta=1-1/p$, while for the latter
one may set $W= \{v\in L^2\;|\; \nabla v\in L^p,\; v \in
L^q\}$ and $\theta=1-1/\max\{p,q\}$. In both cases
Theorem~\ref{mainteo1} applies, while Theorem~\ref{mainteo3}
\emph{cannot} be applied (unless $p=2$). It is an \emph{open
problem}, however, to establish if the last claim of
Theorem~\ref{mainteo3} (passage to the limit in the equation)
still applies when $p\not=2$. A positive answer would settle the
long--standing open question of \emph{global existence} of
weak solutions for this kind of equations (see \cite{Milani}).

\item \emph{Higher order nonlinear equations.} Just to give an
example (see for instance \cite{PT}), consider
\[
\gge(v) = \int \left(\frac 12 |\Delta v|^2 + \frac1p |\nabla v |^p
+ \frac1q |v|^q\right) \,dx \qquad(p,q>1).
\]
Then \eqref{eq::1} becomes the nonlinear vibrating--beam equation
\[
w'' =- \Delta^2 w + \Delta_p w  - |w|^{q-2}w,
\]
where $\Delta^2$ is the biharmonic operator. Here $W= \{v \in
H^2\;|\; \nabla v \in L^p, \; v \in L^q\;\}$, while
$\theta=1-1/\max\{2,p,q\}$. Here Theorem~\ref{mainteo3} applies, and provides global existence.

\item \emph{Kirchhoff equations.} The general scheme presented
in this paper allows one to treat
also \emph{nonlocal} problems. A typical example is the Kirchhoff
equation
\[
w'' = \left (\int |\nabla w|^2\, dx \right) \Delta w.
\]
Here one chooses
\[
\gge(v) = \frac 14\left( \int  |\nabla v|^2\,dx\right)^2,\quad
W=H^1
\]
(note that \eqref{stimagattheta} holds with $\theta=3/4$), and
Theorem~\ref{mainteo1} applies (while it is an \emph{open problem}
to see if the last claim of Theorem~\ref{mainteo3} is true in this
case). More generally, if $\gge(v) = \frac12\Phi \left( \int
|\nabla v|^2\,dx\right)$ for some appropriate function $\Phi$, one
formally obtains the equation
\[
w'' = \Phi'\left (\int |\nabla w|^2\, dx \right) \Delta w 
\]
(the appropriate constant $\theta$ in\eqref{stimagattheta}  will depend on
$\Phi$).

\item \namedlabel{exFL}{\arabic{enumi}}\emph{Wave equations with the fractional Laplacian.} Given
$s\in (0,1)$, we may consider the nonlocal energy
\[
\gge(v)=c\iint \frac{|v(x)-v(y)|^2}{|x-y|^{n+2s}}\,dxdy+\frac\lambda p
\int |v|^p\,dx\qquad (c>0,\,\,\lambda\geq 0,\,\, p>1),
\]
with domain $W=H^s\cap L^p$ (or simply $H^s$, if $\lambda=0$). It
is well known (see e.g.~\cite{HGuide}) that the first integral is
the natural energy associated with the \emph{fractional Laplacian}
$(-\Delta)^s$, so that (for a proper choice of $c$, see \cite{HGuide}) \eqref{eq::1} becomes
\[
w''=(-\Delta)^s w -\lambda |w|^{p-2}w,
\]
a (nonlinear, if $\lambda>0$)  wave equation with the fractional
Laplacian. One may check that assumption (H1) is satisfied, with
$\theta=1/2$ (or $\theta=1-1/\max\{2,p\}$ if $\lambda>0$) in
\eqref{stimagattheta}. Here one may apply Theorem \ref{mainteo3},
thus obtaining global existence.

\item[] The next examples concern \emph{dissipative} equations
with a structure as in \eqref{eq::2}: these are related to the
functional in \eqref{eq:deffepsdis}, as stated in Problem~2. We
will mainly focus on the choice of the functional $\diss$, thus
obtaining dissipative variants of the previous examples.

\item \emph{Telegraph type equations.} These are obtained letting
\[
\diss(v) = \frac12 \int |v|^2\, dx\qquad \text{(with domain
$H=L^2$)},
\]
thus fulfilling assumption (H2). Since
$\langle \nabla\diss(v),\cdot\rangle=\langle v,\cdot\rangle_{L^2}$,
by Remark~\ref{remtypicalH}
this
choice of $\diss$
 generates
the term $-w'$ in the right--hand--side of \eqref{eq::2}.

If, for instance,  $\gge$ is as in Example~\ref{exNLW}, then
we can obtain the nonlinear telegraph equation
\[
w'' =  \Delta w -|w|^{p-2}w -w'.
\]
In this case one can solve Problem~2 completely, since
Theorem \ref{mainteo3} can now be applied with $\kappa=1$.

If, on the other hand, $\gge$ is as in
Example~\ref{exQuasilin}, then one obtains
\[
w'' = \Delta_p w -|w|^{q-2}w -w'
\]
and so on. In fact, in practice, the term $-w'$ can be inserted in
any of the above examples (Theorem~\ref{mainteo2} can always be
applied, while Theorem~\ref{mainteo3} should now be applied with
$\kappa=1$, when possible).

\item \emph{Strongly damped wave equations.} The term ``strongly
damped'' usually denotes the presence of $\Delta w'$ in the
equation (se e.g. \cite{KZ}). We can treat this case by letting
\[
\diss(v) = \frac12 \int |\nabla v|^2\, dx\qquad\text{(with domain
$H=H^1$)}
\]
so that $\nabla\diss(v)$ corresponds to $-\Delta v$ by Remark~\ref{remtypicalH}.
 Then, building on
Example~\ref{exNLW}, we may consider
\[
w'' = \Delta w - |w|^{p-2}w +\Delta w'
\]
(for which Theorem \ref{mainteo3} applies with $\kappa=1$), or
quasilinear versions such as
\[
w'' = \Delta_p w -|w|^{q-2}w +\Delta w'.
\]
The last equation does \emph{not} satisfy the assumptions of
Theorem~\ref{mainteo3} (unless $p=2$). In a forthcoming paper, however, we will
show that the \emph{claim} of Theorem~\ref{mainteo3} is \emph{in
fact true}, for every $p,q>1$.

\item \emph{Other damped equations.} In each of
Examples 1--\ref{exLastND} one can add several dissipative terms. For example,
by Remark~\ref{remtypicalH}, the choice
\[
\diss(v) = \frac12 \int \left(|\Delta v|^2 + |\nabla v|^2 +
|v|^2\right)\, dx
\]
would introduce, in any given equation, the term $-\Delta^2 w' +
\Delta w' - w'$.
\end{list}

\end{document}